\documentclass[a4paper,12pt]{amsart}

\usepackage{youngtab}
\usepackage{enumerate}

\usepackage{t1enc}
\def\frak{\mathfrak}

\def\Cal{\mathcal}

\def\D{\mathbb{D}}

\def\R{\mathbb{R}}

\def\Y{\mathbb{Y}}
\def\Z{\mathbb{Z}}

\def\cA{\mathcal{A}}

\def\cE{\mathcal{E}}
\def\cI{\mathcal{I}}

\def\cT{\mathcal{T}}

\def\cU{\mathcal{U}}

\def\cB{\mathcal{B}}

\def\cK{\mathcal{K}}

\newcommand{\rP}{\mathrm{P}}
\newcommand{\rJ}{\mathrm{J}}

\def\de{\delta}

\def\si{\sigma}

\def\ph{\varphi}

\def\om{\omega}
\def\Ga{\Gamma}
\def\De{\Delta}

\def\na{\nabla}

\def\form#1{\mathbf{#1}}

\newcommand{\lpl}{
  \mbox{$
  \begin{picture}(12.7,8)(-.5,-1)
  \put(2,0.2){$+$}
  \put(6.2,2.8){\oval(8,8)[l]}
  \end{picture}$}}

\newcommand{\upl}
  {\mbox{$
  \begin{picture}(12.7,8)(-.5,-1)
  \put(2,0.2){$+$}
  \put(6.2,2.8){\oval(8,8)[t]}
  \end{picture}$}}

\newlength\celldim
\newlength\fontheight
\newlength\extraheight

\def\myng#1{{\Yvcentermath1\scriptsize\yng(#1{})}}

\newcommand{\newc}{\newcommand}

\newtheorem{theorem}{Theorem}[section]
\newtheorem{lemma}[theorem]{Lemma}
\newtheorem{proposition}[theorem]{Proposition}
\newtheorem{corollary}[theorem]{Corollary}
\newtheorem{definition}[theorem]{Definition}

\theoremstyle{remark}
\newtheorem{remark}[theorem]{\rm\bf Remark}
\newtheorem*{remark*}{\rm\bf Remark}

\newcommand{\ce}{{\Cal E}}

\usepackage{amssymb,stmaryrd}
\usepackage{amscd}

\newcommand{\nd}{\nabla}

\newcommand{\nn}[1]{(\ref{#1})}

\newc{\aR}{\mbox{\boldmath{$ R$}}}
\newc{\aS}{\mbox{\boldmath{$ S$}}}
\newc{\aDeR}{\mbox{\boldmath{$ U$}}_B{}^P{}_C{}^Q}
\newc{\aDe}{\mbox{\boldmath$ \Delta$}}
\newc{\aNd}{\mbox{\boldmath$ \nabla$}}

\newc{\aK}{\mbox{\boldmath{$ K$}}}
\newc{\aL}{\mbox{\boldmath{$ L$}}}

\newcommand{\gog}{\mathfrak g}

\def\sideremark#1{\ifvmode\leavevmode\fi\vadjust{\vbox to0pt{\vss
 \hbox to 0pt{\hskip\hsize\hskip1em
 \vbox{\hsize3cm\tiny\raggedright\pretolerance10000
 \noindent #1\hfill}\hss}\vbox to8pt{\vfil}\vss}}}%
                        
\def\idx#1{{\em #1\/}}

\author{J.-P.\ Michel, P.\ Somberg and  J.\ \v Silhan}
\title{Prolongation of symmetric Killing tensors and 
commuting symmetries of the Laplace operator}

\begin{document}

\begin{abstract}
We determine the space of commuting symmetries of the Laplace operator on
pseudo-Riemannian manifolds of constant curvature, and derive
its algebra structure. Our construction is based
on the Riemannian tractor calculus, allowing to construct a prolongation of 
the differential system
for symmetric Killing tensors. We also discuss some aspects of its relation 
to projective differential geometry.



\end{abstract}

\address{JPM: University of Li\`ege \\ 
Grande Traverse, 12, Sart-Tilman \\ 
B-4000 Li\`ege, Belgium}
\email{jean-philippe.michel@ulg.ac.be}

\address{PS: Mathematical Institute \\ Charles
  University \\ Sokolovsk\'a 83 \\
  Prague\\ Czech Republic} \email{somberg@karlin.mff.cuni.cz}

\address{JS: Institute of Mathematics and Statistics \\ Masaryk
  University \\ Building 08 \\ Kotl\'a\v{r}sk\'a 2 \\ 611 37,
  Brno\\ Czech Republic} \email{silhan@math.muni.cz}

\maketitle

\pagestyle{myheadings}
\markboth{Michel, Somberg \& \v Silhan}{Prolongation, Killing tensors, 
Commuting symmetries of the Laplace operator}


\section{Introduction}

The Laplace operator is one of the cornerstones of geometrical analysis on
pseudo-Riemannian manifolds, and there exists a close relationship between spectral
properties of the Laplace operator and local as well as global invariants of the underlying 
pseudo-Riemannian manifold.

The question of conformal symmetries of the Yamabe-Laplace operator 
$\De$ on conformally flat spaces has been solved in \cite{east}. 
A differential operator $D$ is a conformal symmetry of $\De$ provided 
$[\Delta,D]\in(\Delta)$, where $(\Delta)$ is the left ideal 
generated by $\De$ in the algebra of differential operators.
These operators $D$ are called conformal symmetries, because they 
preserve the kernel of $\De$. On a given flat conformal 
manifold $M$, there is a bijection between the vector space of 
symmetric conformal Killing tensors and the quotient of the space of conformal symmetries by 
$(\Delta)$. Notice that the space of symmetric conformal Killing
tensors is the solution space of a conformally invariant system of overdetermined
partial differential equations, which turns out to be locally finite dimensional.

In the present paper, we classify the commuting symmetries of the Laplace
operator on pseudo-Riemannian manifolds of constant curvature, i.e.,
manifolds locally isometric to a space form. In this case the 
Laplace operator differs from the Yamabe-Laplace operator by a multiple 
of the identity operator. With an abuse of notation, we denote both by $\De$.
 The eigenspaces of the Laplace operator  are 
preserved by commuting symmetries, i.e., 
by linear differential operators $D$ commuting with the Laplace operator:
$$
[\Delta,D]=0.
$$ 
The vector space of commuting symmetries is generated by     
Killing vector fields and their far reaching generalization called
symmetric Killing tensors, or Killing tensors for short.  
Their composition as differential 
operators provides an algebra structure that we determine. 

Killing $2$-tensors on pseudo-Riemannian manifolds are the most studied among Killing tensors, 
and they play a key role in the separation of variables
of the Laplace equation. The construction of commuting symmetries out of 
Killing $2$-tensors is well-known in number of geometrical situations \cite{Car77},
in particular on constant curvature manifolds.
Higher Killing tensors give integrals of motion for the geodesic equation and 
contribute to its integrability. 
They can be regarded as hidden symmetries of the underlying pseudo-Riemannian manifold.
Killing tensors themselves are solutions of an invariant system 
of PDEs, and trace-free Killing tensors are special examples of
conformal Killing tensors.

As a technical tool we introduce, and to a certain extent develop, 
the Riemannian tractor calculus, focusing mainly on manifolds of
constant curvature. 
This allows a uniform description of the prolongation of the 
invariant system of PDEs for Killing tensors, and plays a key role
in our analysis of the correspondence between commuting symmetries of the
Laplace operator and Killing tensors.
In particular, we obtain an explicit version of the identification in \cite{MMS} (see also \cite{Tak})
of the space of fixed valence Killing tensors with a representation
of the general linear group.

The Riemannian tractor calculus can be interpreted 
as the tractor calculus for projective parabolic geometry in a scale corresponding to
a metric connection in the projective class of affine connections.
Restricting to locally flat special affine connections, 
Einstein metric 
connections in the projective class correspond to manifolds of constant
curvature \cite{gm}. 
Since the Killing equations on symmetric tensor fields are
projectively invariant \cite{EGproj}, 
we can use the invariant tractor calculus in projective parabolic geometry
to construct commuting symmetries. Note the projective invariance
explains that the space of Killing tensors carries a representation of the general linear group.

As for the style of presentation and exposition, we have tried to make the paper accessible 
to a broad audience, with basic knowledge in Riemannian geometry. 
Following this perspective, the structure of our paper goes as follows. After setting the conventions in 
Section $2$, we introduce in Section $3$ the rudiments of Riemannian tractor calculus.
The core of the article is in Section $4$, where we construct the 
prolongation of the differential system for Killing tensors and 
derive the space of differential operators preserving the spectrum of the Laplace operator.
Afterwards, we determine the underlying structure of associative algebra on this space, 
induced by the composition of differential operators.
We compute explicit formulas for commuting symmetries of order at most $3$.
In special cases we compare commuting symmetries with conformal symmetries 
constructed in \cite{east}. 
In Section $5$, we interpret our results in terms of the holonomy reduction
of a Cartan connection in projective parabolic geometry and its restriction on a
curved orbit equipped with Einstein metric. 

The authors are grateful to A.\ \v Cap, M.\ G.\ Eastwood and A.\ R.\ Gover for 
their suggestion  on the interpretation of our results in the framework of projective
parabolic geometry.

\section{Notation and conventions}

Let $(M,g)$ be a smooth pseudo-Riemannian manifold. 
Throughout the paper we employ the Penrose's abstract index
notation and write $\ce^a$ to denote the space of smooth
sections of the tangent bundle $TM$ on $M$, and $\ce_a$ for the space
of smooth sections of the cotangent bundle $T^*M$. 
We also write $\ce$
for the space of smooth functions. All tensors considered are
assumed to be smooth.
 With abuse of 
notation, we will often use the same symbols for the bundles and their
spaces of sections. 
The metric $g_{ab}$ will be used to identify $TM$ with
$T^*M$. We shall assume that the manifold $M$ has dimension $n\geq 2$.

An index which appears twice,
once raised and once lowered, indicates the contraction.  
The square brackets $[\cdots]$ will denote skew-symmetrization of enclosed indices,
while the round brackets $(\cdots)$ will indicate symmetrization.

We write $\nabla$ for the Levi-Civita connection corresponding to $g_{ab}$. 
Then the Laplacian $\Delta$ is given by
$\Delta=g^{ab}\nd_a\nd_b= \nd^b\nd_b\,$. Since the Levi-Civita connection is 
torsion-free, the Riemannian curvature 
$R_{ab}{}^c{}_d$ is given by $ [\nd_a,\nd_b]v^c=R_{ab}{}^c{}_dv^d $, where 
$[\cdot,\cdot]$ indicates the commutator bracket.  The Riemannian
curvature can be decomposed in terms of the totally trace-free 
Weyl curvature $C_{abcd}$ and the symmetric
Schouten tensor $\rP_{ab}$,
\begin{eqnarray} \label{Rriem}
R_{abcd}=C_{abcd}+2{ g}_{c[a}\rP_{b]d}+2{ g}_{d[b}\rP_{a]c}.
\end{eqnarray}
We will refer to $\rP_{ab}$ as to Riemannian Schouten
tensor to distinguish from the projective Schouten tensor introduced later. 
We define $\rJ := \rP^a{}_a$, so that $\rJ = \frac{\mathrm{Sc}}{2(n-1)}$ with 
$\mathrm{Sc}$ the scalar curvature. 

Throughout the paper we work (if not stated otherwise) on manifolds of 
constant curvature, i.e., locally symmetric spaces with parallel curvature 
$R_{abcd} = \frac{4}{n}\rJ g_{c[a}g_{b]d}$, cf. \cite{wo}. 
This means that $\rJ$ is parallel for the Levi-Civita connection $\nabla$. 
In signature $(p,q)$, $M$ is then locally isomorphic to $G/H$, where $G=SO(p+1,q)$ and
$H=SO(p,q)$ if $J>0$, $G=SO(p,q+1)$ and
$H=SO(p-1,q+1)$ if $J<0$, $G=E(p,q)$ and $H=SO(p,q)$ if 
$J=0$. Here we denote by $E(p,q)$ the group of pseudo-Euclidean 
motions on $\R^{p,q}$.

\section{Tractor calculus in Riemannian geometry}

The notion of associated tractor bundles is well-known in the category
of parabolic geometries. We refer to \cite{CSbook} for a review with 
many applications. In the present section we introduce and develop 
rudiments of a class of tractor bundles 
in the category of pseudo-Riemannian manifolds, in 
a close analogy with tractor calculi in parabolic geometries.

Recall we assume $M$ has constant curvature, i.e.\
$R_{abcd} = \frac{4}{n}\rJ g_{c[a}g_{b]d}$ with $\rJ$ constant.
We define the \idx{Riemannian standard tractor bundle} (standard tractor 
bundle, for short) as 
$\cT := \mathcal{L} \oplus TM$ where $ \mathcal{L}$ denotes the trivial bundle 
over $M$. 

The Levi-Civita connection $\nabla_a$ induces a connection on $\cT$, 
which is trivial on $\mathcal{L}$.
The \idx{tractor connection} is another connection on $\cT$,
 also denoted (with an abuse of notation) by $\nabla_a$, and defined by 
\begin{equation}\label{stdtractor}
\nabla_a\begin{pmatrix}
f \\ \mu^b
\end{pmatrix}
=\begin{pmatrix}
\nabla_a f  - \mu_a \\
\nabla_a \mu^b + \frac2n \rJ f \de_a^b
\end{pmatrix}
\end{equation}
where $f \in \cE$, $\mu^b \in \cE^b$. In the first line, we use the isomorphism $TM \cong T^*M$. 
The dual connection on  
the dual bundle $\cT^* := T^*M \oplus \mathcal{L}$, denoted also by
$\na_a$, is given by
\begin{equation}\label{stdtractor*}
\nabla_a\begin{pmatrix}
\nu_b \\ f
\end{pmatrix}
=\begin{pmatrix}
\nabla_a \nu_b +f g_{ab}  \\
\nabla_a f - \frac2n \rJ \nu_a
\end{pmatrix}
\end{equation}
where $\nu_b \in \cE_b$ and $f \in \cE$. A direct computation shows that the 
curvature of the tractor connection $\na$ is trivial, i.e., the tractor connection 
$\na$ is flat. Note this connection differs
from the tractor connection induced by the Cartan connection, as discussed
in \cite[section 1.5]{CSbook}. 

The bundle $\cT$ is equipped with the symmetric bilinear form 
$\langle, \rangle$,
\begin{eqnarray} \label{stdpairing}
\langle
\begin{pmatrix}
f \\ \mu^b
\end{pmatrix},
\begin{pmatrix}
\bar{f}\\ \bar{\mu}^b
\end{pmatrix} \rangle = 
\frac{2}{n} \rJ f\bar{f} + \mu^b \bar{\mu}_b 
\end{eqnarray}
which is invariant with respect to the tractor connection $\na$. For 
$\rJ \not=0$, this form is non-degenerate and called the 
\idx{tractor metric}. It yields then an isomorphism $\cT \cong \cT^*$.

We define the \idx{adjoint tractor bundle} as
$\cA := \bigwedge^2 \cT = TM \oplus \wedge^2 TM$ and we extend the tractor
connection from $\cT$ to $\cA$ by the Leibniz rule. Similarly, we obtain
an induced tractor connection on $\cA^* = \wedge^2 T^*M \oplus T^*M$.
Explicitly, these connections are given by the formulas
\begin{equation} \label{traccon}
\nabla_a\begin{pmatrix}
\varphi^b\\ \psi^{bc}
\end{pmatrix}
=\begin{pmatrix}
\nabla_a \varphi^b  - 2 \psi_a{}^b \\
\nabla_a \psi^{bc} + \frac{2}{n} \rJ \de_a^{[b} \varphi^{c]}
\end{pmatrix}
\quad \text{and} \quad
\nabla_a\begin{pmatrix}
\mu_{bc}\\ \om_{b}
\end{pmatrix}
=\begin{pmatrix}
\nabla_a \mu_{bc}  + 2g_{a[b}\om_{c]} \\
\nabla_a \om_b - \frac{4}{n} \rJ \mu_{ab}
\end{pmatrix},
\end{equation}
where $(\ph^b,\psi^{bc}) \in \Ga(\cA)$, i.e.\ $\ph^b \in \cE^b$ and
$\psi^{bc} \in \cE^{[bc]}$, and 
$(\mu_{bc}, \om_{b}) \in \Ga(\cA^*)$, i.e.\ $\mu_{bc} \in \cE_{[bc]}$ and
$\om_b \in \cE_b$. 

We extend the tractor connection to the tensor product bundle 
$\bigl( \bigotimes \cT \bigr) \otimes \bigl( \bigotimes \cT^* \bigr)$ by the
Leibniz rule. The resulting connection is again flat, denoted by $\na$ and
termed the tractor connection. 
Further, the tractor and Levi-Civita connections induce the connection
\begin{eqnarray}
\na_a: \cE_{b \ldots d} \otimes \Ga(W) 
\to \cE_{ab \ldots d} \otimes \Ga(W)
\end{eqnarray}
for any tractor subbundle $W \subseteq 
\bigl( \bigotimes \cT \bigr) \otimes \bigl( \bigotimes \cT^* \bigr)$, i.e., any subbundle preserved 
by the tractor connection.
This coupled Levi-Civita--tractor connection allows to extend all natural operators, e.g.\ the
Laplace operator $\De$, to tensor-tractor bundles.

The invariant pairing on $\cA$ induced by \nn{stdpairing}
is given by the formula
\begin{eqnarray} \label{adpairing}
\langle
\begin{pmatrix}
\varphi^b\\ \psi^{bc}
\end{pmatrix},
\begin{pmatrix}
\bar{\varphi}^b\\ \bar{\psi}^{bc}
\end{pmatrix} \rangle = 
\frac{1}{n} \rJ \ph^a\bar{\ph}_a + \psi^{ab} \bar{\psi}_{ab}.
\end{eqnarray}
For $\rJ \not=0$, this defines a metric on $\cA$  and $\cA \cong \cA^*$. 
Moreover, there is a Lie algebra structure
$[\cdot,\cdot]: \cA \otimes \cA \to \cA$ given by 
\begin{eqnarray} \label{bracket}
\left[
\begin{pmatrix}
\varphi^b\\ \psi^{bc}
\end{pmatrix},
\begin{pmatrix}
\bar{\varphi}^b\\ \bar{\psi}^{bc}
\end{pmatrix} \right] = 
\begin{pmatrix}
\varphi_r \bar{\psi}^{ra} - \bar{\varphi}_r \psi^{ra} \\
-2 \psi^{r[b} \bar{\psi}_r{}^{c]} - \frac{1}{n}\rJ \varphi^{[b} \bar{\varphi}^{c]} 
\end{pmatrix}
\end{eqnarray}
which is also invariant under the tractor connection. 

\begin{remark}\label{Rmk:Killingvf}
Tractor connections can be defined on any Riemannian manifold. For example
for  $\varphi^b\in \cE^b$ and $\psi^{bc} \in \cE^{[bc]}$, we can define $\na$ 
on $\cA$ by
\begin{equation*}
\nabla_a\begin{pmatrix}
\varphi^b\\ \psi^{bc}
\end{pmatrix}
=\begin{pmatrix}
\nabla_a \varphi^b  - 2 \psi_a{}^b \\
\nabla_a \psi^{bc} + \frac12 R^{bc}{}_{as} \varphi^s
\end{pmatrix}.
\end{equation*}
This definition originates in the work by  B.\ Kostant \cite{Kostant}. 
We observe that, for a Killing vector field 
$k^a \in \cE^a$, its prolongation
\begin{eqnarray}\label{Prol-killing}
K = \begin{pmatrix}
k^a \\ \frac12 \na^{[a}k^{b]}
\end{pmatrix} \in \Ga(\cA)
\end{eqnarray}
is parallel for the tractor connection. Hence, any isometry is 
locally determined by its first jet. 
\end{remark}

We shall use the abstract index notation for the adjoint tractor bundle as
follows: $\Ga(\cT)$ will be denoted by $\cE^A$ and 
$\Ga(\cA)$ will be denoted by $\cE^\form{A}$ where $\form{A} = [A^1A^2]$.
Similarly $\cE_A = \Ga(\cA^*)$ and $\cE_\form{A} = \Ga(\cA^*)$. 
That is, we use bold face capital indices for adjoint tractor bundle and its 
dual.

There is a convenient way to treat bundles $\cT$ and $\cT^*$, based on the 
so-called \idx{injectors}, or tensor-tractor frame, denoted by 
$Y^A, Z^A_a$ for $\cT$ and denoted by $Y_A, Z^a_A$ for $\cT^*$. 
These are defined by
\begin{eqnarray} \label{YZstd}
& \begin{pmatrix}
f \\ \mu^b
\end{pmatrix}
=Y^A f + Z_b^A \mu^b\, , \quad
\begin{pmatrix}
\nu_b \\ f
\end{pmatrix}=
Z_A^b \nu_b
+ Y_Af
\end{eqnarray}
and their contractions are $Y^AY_A=1$, $Z^A_a Z_A^b = \de_a^b$ and 
$Y^A Z_A^b = Z^A_a Y_A =0$.
The covariant derivatives in \nn{stdtractor} and \nn{stdtractor*} are then 
encoded in covariant derivatives of these injectors,
\begin{eqnarray} \label{tracconYZstd}
\begin{split} 
&\nabla_c Y^A=\frac{2}{n}\rJ Z_a^A\delta^{a}_c, 
&&&& \nabla_c Z_a^A= - Y^A g_{ac}^{}, \\
&\nabla_c Z^a_A=
- \frac{2}{n}\rJ Y_A \delta^{a}_c, 
&&&& \nabla_c Y_A = Z^a_A g_{ca}.
\end{split}
\end{eqnarray}
We denote the tractor pairing \nn{stdpairing} by 
$h_{AB} \in \cE_{(AB)}$ which has the explicit form 
\begin{eqnarray}
& h_{AB}=
\frac{2}{n}\rJ\, Y_A Y_B
+ Z^a_A Z^b_B g_{ab}\, .
\end{eqnarray}

Injectors for the adjoint tractor bundle $\cE^\form{A}$ are 
$\Y^{\bf A}_{\,a} = Y^{[A^1}Z^{A^2]}_{\,a}$ and 
$\Z^{\bf A}_{\,\form{a}} = Z^{[A^1}_{\,a^1}Z^{A^2]}_{\,a^2}$, 
injectors for the dual bundle $\cE_\form{A}$ are 
$\Y_\form{A}^{\,a} = Y_{[A^1}^{} Z_{A^2]}^{\,a}$ and 
$\Z_{\bf A}^{\,\form{a}} = Z_{[A^1}^{\,a^1} Z_{A^2]}^{\,a^2}$. That is,
\begin{eqnarray} \label{YZ}
& \begin{pmatrix}
\varphi^b \\ \psi^{\form{a}}
\end{pmatrix}
=\varphi^b\Y^{\bf A}_b+\psi^{\form{a}}\Z_{\,\form{a}}^{\bf A}\, \quad
\begin{pmatrix}
\mu_{\form{a}} \\ \om_b
\end{pmatrix}=
\mu_{\form{a}}\Z^{\,\bf a}_{\bf A}
+ \om_b\Y^{\, b}_{\bf A}, 
\end{eqnarray}
where $\form{a} = [a^1a^2]$. The only nonzero contractions are
$\Y^{\bf A}_{\,a} \Y_\form{A}^{\,b} = \frac12 \de_a^b$ and 
$\Z^{\bf A}_{\,\form{a}} \Z_{\bf A}^{\,\form{c}} = 
\de_{[a^1}^{c^1} \de_{a^2]}^{c^2}$. The covariant derivatives \nn{traccon} 
are then equivalent to
\begin{eqnarray} \label{tracconYZ}
\begin{split} 
&\nabla_c\Y_{\, b}^{\bf A}=\frac{2\rJ}{n}\Z_{ab}^{\bf A}\delta^{a}_c, 
&&&& \nabla_c\Z_{\,\form{a}}^{\bf A}=-2\Y_{[a^2}^{\bf A}g_{a^1]c}^{}, \\
&\nabla_c\Z^{\,\form{a}}_{\bf A}=
-\frac{4\rJ}{n}\Y^{[a^2}_{\bf A}\delta^{a^1]}_c, 
&&&& \nabla_c\Y^{\,b}_{\bf A}=\Z^{ab}_{\bf A}g_{ca}.
\end{split}
\end{eqnarray}
and the pairing  \nn{adpairing} on $\cE^\form{A}$ can be written as
\begin{eqnarray} \label{tracmetricforms}
& h_{\form{A}\form{B}}=
\frac{4}{n}\rJ\, \Y^{\,a}_{\form{A}}\Y^{\,b}_\form{B}g_{ab}
+\Z^{\,\form{a}}_{\form{A}} \Z^{\form{b}}_{\form{B}} g_{a^1b^1} g_{a^2b^2}\, .
\end{eqnarray}

A crucial ingredient in our construction will be the differential operator
\begin{equation} \label{doubleD}
\D_\form{A}: \cE_{b_1 \ldots b_s} \otimes \Ga(W) 
\to \cE_{b_1 \ldots b_s} \otimes \Ga(\cA^* \otimes W)
\end{equation}
for a tractor subbundle 
$W \subseteq \bigl( \bigotimes \cT \bigr) \otimes \bigl( \bigotimes \cT^* \bigr)$.
This operator is closely related to the so-called fundamental derivative \cite{CSbook}.
It is defined as follows: for $f \in \Gamma(W)$, we put
\begin{equation} \label{doubleD1}
\quad \D_{\bf A} f =  \begin{pmatrix}
0 \\ 2\na_a f
\end{pmatrix} \in \Ga(\cA^* \otimes W), 
\end{equation}
and for $\varphi_b \in \cE_b$, we put 
\begin{equation} \label{doubleD2}
\D_{\bf A} \varphi_b = \begin{pmatrix}
2g_{b[a^1}\varphi_{a^2]} \\ 2\na_a \varphi_b
\end{pmatrix} \in \cE_b \otimes \Ga(\cA^*).
\end{equation}
Then we extend $\D_{\bf A}$ to all tensor-tractor bundles by the Leibniz rule.
Using injectors \nn{YZ}, the formulas \nn{doubleD1} and \nn{doubleD2} 
are given by 
\begin{eqnarray} \label{doubleD12}
\begin{split}
& \D_{\bf A}f=2\Y_{\bf A}^{\,a} \nabla_af, \\
& \D_{\form{A}}\varphi_b=2\Y_{\bf A}^{\,a}\nabla_a\varphi_b
+\Z_{\bf A}^{\,\form{a}}2g_{b[a^0}\varphi_{a^1]}
\end{split}
\end{eqnarray}
where $\form{a} = [a^1a^2]$. 

\begin{theorem} \label{Dcomm}
Let $M$ be a manifold of constant curvature. 
The operator $\D_{\bf A}$ commutes with 
the coupled Levi-Civita--tractor connection~$\nabla_c$, 
$$
\nabla_c \D_{\bf A} = \D_{\bf A}\nabla_c: 
\cE_{b_1 \ldots b_s} \otimes \Ga(W) \to 
\cE_{c{\form{A}}b_1 \ldots b_s} \otimes \Ga(W).
$$
\end{theorem}

\begin{proof}
Since the tractor connection is flat, it is sufficient
to prove the statement for $W$ equal to the trivial line bundle. We present two versions 
of the proof. 

First, one can easily show by direct
computation (using \nn{tracconYZ} and \nn{doubleD12}) that the
explicit formulas for the compositions $\nabla_c \D_{\bf A}$ and $\D_{\bf A}\nabla_c$ 
are the same when acting on $f\in \cE$ and 
$\varphi_b \in \cE_b$. Hence the formulas agree on any tensor bundle.

Alternatively, recall that for a Killing vector field $k^a \in \cE^a$, its prolongation 
$K^{\bf A} \in \Ga(\cA)$ is parallel, see \eqref{Prol-killing}. 
We further observe that 
$L_k =  K^{\bf A} \D_{\bf A}$ is the Lie derivative along $k^a$ when acting on 
tensor bundles. Since $L_k$ commutes with the covariant derivative and the 
space of Killing vector fields on manifolds with constant curvature 
has dimension equal to $\dim(\cA) = n + \frac12 n(n-1)$, the statement follows.
\end{proof}

As a consequence of the previous theorem, $\D_{\bf A}$ commutes also with the Laplace
operator $\De$ on functions, forms, etc. Let us make this result more general and more precise.
Assume $F: \Ga(U_1) \to \Ga(U_2)$ is a Riemannian invariant linear differential 
operator, acting between tensor bundles $U_1$, $U_2$. Then, it can be written in terms
of the metric, the Levi-Civita connection $\na$ and the curvature $\rJ$. 
Regarding $\na$ in the formula for $F$ as the 
coupled Levi-Civita--tractor connection, we obtain the operator 
$F^\na: \Ga(\cA^* \otimes U_1) \to \Ga(\cA^* \otimes U_2)$. 
Using Theorem \ref{Dcomm} and $\nabla_a\rJ=\nabla_a g=0$, we get the following

\begin{corollary} \label{dDcomm}
Let $F: \Ga(U_1) \to \Ga(U_2)$ be a Riemannian invariant linear differential 
operator on the manifold $M$. Then $\D_{\bf A}$ commutes with $F$, i.e.,
$$
\D_{\bf A} \circ F = F^\na \circ \D_{\bf A}: \Ga(U_1) \to \Ga(\cA^* \otimes U_2).
$$
\end{corollary}


\section{Commuting symmetries of the Laplace operator}
\begin{definition}
Let $U$ be a tensor bundle and $F: \Ga(U) \to \Ga(U)$ be a linear 
differential operator on $M$. A \idx{commuting symmetry} of the operator $F$ 
is a linear differential operator ${\fam2 D}$ fulfilling 
${\fam2 D}F=F{\fam2 D}$.
\end{definition}

We are interested in commuting symmetries of the Laplace operator $F=\De$
on functions. They form a subalgebra of the associative algebra of linear differential 
operators acting on $\cE$.
The commuting symmetries are particular case of the conformal symmetries 
of $\De$, studied in \cite{east}. 
The exposition in the rest of the
section closely follows the one given in \cite{east}.

Let $\ell$ be a non-negative integer.
A linear $\ell$-th order differential operator acting on functions can be written as 
\begin{eqnarray}\label{gfdo}
{\fam2 D}=V^{a_1\dots a_\ell}\nabla_{a_1}\dots\nabla_{a_\ell}+LOTS,
\end{eqnarray}
where "LOTS" stands for lower order terms in ${\fam2 D}$ and 
its principal symbol $V^{a_1\dots a_\ell}$ is symmetric in its indices,
$V^{a_1\dots a_\ell}=V^{(a_1\dots a_\ell)}$.

\begin{definition}
A Killing tensor on $M$ is a symmetric tensor field 
$V^{a_1\dots a_\ell}$, 
fulfilling the first order differential equation 
\begin{eqnarray}\label{kte}
\nabla^{(a_0}V^{a_1\dots a_\ell)}=0.  
\end{eqnarray}  
The vector space of all Killing tensors of valence $\ell$ will be denoted 
${\cK}_{\ell}$.  
\end{definition}
Since the differential equation (\ref{kte}) is overdetermined, the space
$\cK_\ell$ is finite dimensional. Note that the symmetric product of two
Killing tensors is again a Killing tensor, so that $\bigoplus_{\ell\leq 0}\cK_\ell$
is a commutative graded algebra.

\begin{theorem}\label{th1}
Let ${\fam2 D}$ be a $\ell$-th order commuting symmetry of the Laplace 
operator. Then the principal symbol $V^{a_1\dots a_\ell}$ of ${\fam2 D}$ is a
Killing tensor of valence $\ell$.
\end{theorem}
\begin{proof}
When ${\fam2 D}$ is of the form (\ref{gfdo}), we compute
\begin{eqnarray}
\De{\fam2 D}-{\fam2 D}\De =2(\nabla^bV^{a_1\dots a_\ell})\nabla_b\nabla_{a_1}\dots\nabla_{a_\ell}
+"LOTS",
\end{eqnarray}
and the claim follows.
\end{proof}

The converse statement is the content of the following Theorem. 

\begin{theorem}\label{th2}
There exists a linear map,
$V^{a_1\dots a_\ell}\mapsto {\fam2 D}^{V}$, 
from  symmetric tensor fields to differential operators,
such that the principal symbol of ${\fam2 D}^{V}$ is $V^{a_1\dots a_\ell}$ 
and ${\fam2 D}^{V}\De=\De{\fam2 D}^{V}$ if $V$ is a Killing tensor.
\end{theorem}

The proof of the Theorem is postponed to the next section, where 
the Riemannian prolongation connection for symmetric powers of the adjoint 
tractor bundle is constructed. This allows to compute explicitly 
the symmetry operators ${\fam2 D}^V$.

Combining both theorems, we deduce a linear bijection between the space
of Killing tensors and the space of commuting symmetries of $\De$.
In particular, the dimension of the vector space 
of $\ell$-th order symmetry operators is finite, equal to $\dim\cK_0+\dim\cK_1+\cdots+\dim \cK_\ell$. 


\subsection{Prolongation for Killing tensors.}
Let $k^a \in \cE^a$ be a Killing vector field. In Remark \ref{Rmk:Killingvf}, 
we observed that its prolongation $K = (k^a, \frac{1}{2} \na^{[a}k^{b]}) \in 
\Ga(\cA)$ 
is parallel for the tractor connection.
Our aim is to construct analogous prolongation for Killing tensors.

\begin{lemma} \label{identity}
If $k^{a_1 \ldots a_\ell}\in\cE^{(a_1 \ldots a_\ell)}$ is a Killing tensor then
\begin{eqnarray}
\na^{(a_1} \na^{|[c} k^{d]|a_2 \ldots a_\ell)} = 
-\frac{2(\ell+1)}{n} \rJ \, g^{(a_1|[c} k^{d]|a_2 \ldots a_\ell)},
\end{eqnarray}
where the notation $|\ldots|$ means the enclosed indices $c$, $d$ 
are excluded from the symmetrization.
\end{lemma}

\begin{proof}
A straightforward computation.
\end{proof}

The prolongation of $k^{a_1 \ldots a_\ell}$ will be a section 
$K \in \Ga(\otimes^\ell\cA)$.  As a first step, we observe the following:
\begin{lemma}
The differential operator 
\begin{eqnarray} \label{Pi}
\begin{split}
& \Pi: \cE^{(a_1 \ldots a_\ell)} \to \cE^{(a_1 \ldots a_{\ell-1})} 
\otimes \Ga(\cA), \quad \\
& (\Pi \si)^{a_1 \ldots a_{\ell-1}\form{B}} =
\begin{pmatrix}
\si^{c a_1 \ldots a_{\ell-1}}  \\ 
\frac{1}{\ell+1} \na^{[c}\si^{d]a_1 \ldots a_{\ell-1}}
\end{pmatrix} \in \cE^{(a_1 \ldots a_{\ell-1})} \otimes \Ga(\cA)
\end{split}
\end{eqnarray}
satisfies, for all  $\si^{a_1 \ldots a_{\ell}}\in\cE^{(a_1 \ldots a_\ell)}$,
\begin{equation} \label{step}
\na^{(a_0} \si^{a_1 \ldots a_{\ell})} = 0 \Longleftrightarrow 
\na^{(a_0} (\Pi \si)^{a_1 \ldots a_{\ell-1})\form{B}}=0.
\end{equation}
\end{lemma}

\begin{proof}
Using \nn{traccon}, we compute
$$
\na^{b} (\Pi \si)^{a_1 \ldots a_{\ell-1}\form{B}} = 
\begin{pmatrix}
\na^{b} \si^{ca_1 \ldots a_{\ell-1}}
- \frac{2}{\ell+1} \na^{[b} \si^{c]a_1 \ldots a_{\ell-1}} \\ 
\frac{1}{\ell+1} \na^b \na^{[c}\si^{d]a_1 \ldots a_{\ell-1}}
+ \frac{2}{n} \rJ \, g^{b[c} \si^{d] a_1 \ldots a_{\ell-1}}
\end{pmatrix}.
$$
Observe the "top slot" on the right hand side is equal to 
$\frac{\ell}{\ell+1} \na^b \si^{c a_1 \ldots a_{\ell-1}}
+ \frac{1}{\ell+1} \na^{c} \si^{ba_1 \ldots a_{\ell-1}}$, which
after the symmetrization $(ba_1 \ldots a_{\ell-1})$ yields exactly
$\na^{(c} \si^{b a_1 \ldots a_{\ell-1})}$. This proves the implication
$\Longleftarrow$ of \nn{step}, and also that if 
$\na^{(a_0} \si^{a_1 \ldots a_{\ell})} =0$ then the "top slot" of 
$\na^{(a_0} (\Pi \si)^{a_1 \ldots a_{\ell-1})}$ vanishes. Since the bottom slot 
vanishes by Lemma \ref{identity}, the implication $\Longrightarrow$
in \nn{step} follows as well.
\end{proof}

Considering $\na$ in the formula \eqref{Pi} as the coupled Levi-Civita--tractor 
connection, we obtain the operator
$\Pi: \cE^{(a_1 \ldots a_\ell)} \otimes \Ga(W) \to 
\cE^{(a_1 \ldots a_{\ell-1})} \otimes \Ga(\cA \otimes W)$,
where $W \subseteq \bigl( \bigotimes \cT \bigr) \otimes 
\bigl( \bigotimes \cT^* \bigr)$ is a tractor subbundle. 
Its iteration 
\begin{equation}\label{Piell}
\Pi^{(\ell)}:\cE^{(a_1 \ldots a_\ell)}\rightarrow \Ga(\otimes^\ell\cA)
\end{equation}
yields the prolongation for Killing tensors.
\begin{proposition} \label{prolong}
Let $\si^{a_1 \ldots a_\ell} \in \cE^{(a_1 \ldots a_\ell)}$. 
The operator $\Pi^{(\ell)}$ satisfies
\begin{eqnarray}\label{Eq-prolongation}
\na^{(a_0} \si^{a_1 \ldots a_{\ell})} = 0 \Longleftrightarrow 
\na \left(\Pi^{(\ell)} \si\right)=0.
\end{eqnarray}
\end{proposition}

\begin{proof}
Since the tractor connection is flat, we have the analogue of \nn{step}:
$$
\na^{(a_0} \si^{a_1 \ldots a_{\ell})\bullet} = 0 \Longleftrightarrow 
\na^{(a_0}  (\Pi \si)^{a_1 \ldots a_{\ell-1})\bullet}=0,
$$
for all $\si^{a_1 \ldots a_\ell \bullet} \in \cE^{(a_1 \ldots a_\ell)} 
\otimes \Ga(W)$ where  $\bullet$ denotes an unspecified tractor index.
By iteration, we get then \eqref{Eq-prolongation}.
\end{proof}
The symmetries of the tractor $\Pi^{(\ell)} \si$ are best understood in the language of Young diagrams.
Putting $\cT = \myng{1}$, we have $\cA = \myng{1,1}$ and we set
\begin{equation}\label{boxtimesA}
\boxtimes^\ell\cA := 
\raisebox{-13pt}{$
\overbrace{\begin{picture}(60,30)
\put(0,5){\line(1,0){60}}
\put(0,15){\line(1,0){60}}
\put(0,25){\line(1,0){60}}
\put(0,5){\line(0,1){20}}
\put(10,5){\line(0,1){20}}
\put(20,5){\line(0,1){20}}
\put(50,5){\line(0,1){20}}
\put(60,5){\line(0,1){20}}
\put(35,10){\makebox(0,0){$\cdots$}}
\put(35,20){\makebox(0,0){$\cdots$}}
\end{picture}}^{\ell}$}
\subseteq  S^\ell \cA,
\end{equation}
where $S^\ell\cA\subset \bigotimes^\ell\cA$ is the subspace of symmetric tensors.
\begin{proposition}\label{PiBoxtimes}
The map $\Pi^{(\ell)}$, defined in \nn{Piell}, is valued in $\Ga(\boxtimes^\ell\cA)$.
\end{proposition}
\begin{proof}
Let $\si^{a_1 \ldots a_\ell} \in \cE^{(a_1 \ldots a_\ell)}$. 
Using abstract indices $\form{B}_i = [B_i^1B_i^2]$, we have
$\bigl( \Pi^{(\ell)}\sigma \bigr)^{B_1^1 B_1^2 \ldots B_\ell^1 B_\ell^2}
\in \cE^{\form{B}_1 \ldots \form{B}_\ell}=\Ga(\otimes^\ell\cA)$.

First, we prove that 
$(\Pi^{(\ell)} \si)^{\form{B}_1 \ldots \form{B}_\ell}$ is symmetric in indices
$\form{B}_1$, \ldots, $\form{B}_\ell$, i.e $\Pi^{(\ell)} \si\in\Ga(S^\ell \cA)$. 
In fact, it is sufficient to
show the symmetry in two neighboring indices $\form{B}_i$, $\form{B}_{i+1}$.
To do this, we show that, for all  
$\si^{a_1 \ldots a_\ell \bullet} \in \cE^{(a_1 \ldots a_\ell)\bullet}$ 
with $\ell \geq 2$, 
$(\Pi^{(2)} \si)^{a_1 \ldots a_{\ell-2}\form{B}_1\form{B}_2\bullet}$ is 
symmetric in $\form{B}_1$ and $\form{B}_2$ . 
This follows from the explicit formula 
\begin{eqnarray} \label{formula}
\begin{split}
(\Pi^{(2)} &\si)^{a_1 \ldots a_{\ell-2}\form{B}\form{C}\bullet} = 
\Y^\form{B}_{\,b} \Y^\form{C}_{\,c} \si^{a_1 \ldots a_{\ell-2}bc\bullet} \\
& + \frac{1}{\ell+1} \left[ 
\Y^\form{B}_{\,b} \Z^\form{C}_{\,\form{c}} 
\na^{c^1} \si^{a_1 \ldots a_{\ell-2}bc^2\bullet}
+\Z^\form{B}_{\,\form{b}} \Y^\form{C}_{\,c} 
\na^{b^1} \si^{a_1 \ldots a_{\ell-2}b^2c\bullet} \right] \\
& + \frac{1}{\ell}  
\Z^\form{B}_{\,\form{b}}  \Z^\form{C}_{\,\form{c}} \left[ \frac{1}{\ell+1}
\na^{b^1} \na^{c^1} \si^{a_1 \ldots a_{\ell-2}b^2c^2\bullet}
+ \frac2n \rJ g^{b^1c^1} \si^{a_1 \ldots a_{\ell-2}b^2c^2\bullet} \right]
\end{split}
\end{eqnarray}
obtained from \nn{Pi} after a short computation.
Here $\form{B} = [B^1B^2]$, $\form{C} = [C^1C^2]$, $\form{b} = [b^1b^2]$ and
$\form{c} = [c^1c^2]$.

Then, it suffices to prove that
$\bigl( \Pi^{(\ell)}\sigma \bigr)^{B_1^1 B_1^2 \ldots B_\ell^1 B_\ell^2}
\in \cE^{\form{B}_1 \ldots \form{B}_\ell}$ 
vanishes after skew--symmetrization over any triple of indices $B_i^j$. 
Since $(\Pi^{(\ell)} \si)^{\form{B}_1 \ldots \form{B}_\ell}$ is symmetric in 
tractor form indices $\form{B}_i$, it is sufficient to consider only two triples
of indices: either $B_1^1$, $B_1^2$, $B_2^1$ or $B_1^1$, $B_2^1$, $B_3^1$.
Elementary representation theory shows that 
the third symmetric power of $\cA$ has the 
decomposition
\begin{equation} \label{3sym}
S^3 \bigl( \myng{1,1} \bigr) =  \myng{1,1,1,1,1,1} \oplus
2\,  \myng{2,2,1,1} \oplus  \myng{3,3}.
\end{equation}
Hence, if we skew over three 
factors of the standard tractor bundle in $S^3 \cE^{\form{A}}$, 
the result will be in fact skew 
symmetric in at least four factors of the standard tractor bundle. 
As a result, it is sufficient to consider only skew symmetrization over indices 
$B_1^1$, $B_1^2$, $B_2^1$ of 
$\bigl( \Pi^{(\ell)} \si \bigr)^{B_1^1 B_1^2 \ldots B_\ell^1 B_\ell^2}$.
It is a straightforward computation to show 
$\bigl( \Pi^{(\ell)} \si\bigr)^{[B_1^1 B_1^2 B_2^1]B_2^2 \ldots 
B_\ell^1 B_\ell^2} =0$, using \nn{formula} with $\ell=2$.
\end{proof}

The proposition can be also proved using invariant techniques in 
parabolic geometries (known as "BGG machinery") applied to the case of projective 
parabolic geometry. We refer to the next section for further discussion on this 
relation.

Next, we get the main result of this section
\begin{theorem}\label{Cor:KT-Young}
The map $\Pi^{(\ell)}$ induces a bijective correspondence between 
the space $\cK_\ell$ of Killing $\ell$-tensors and
the space of parallel sections of the tractor bundle $\boxtimes^\ell\cA$.
\end{theorem}
\begin{proof}
If $\sigma_\ell$ is a Killing $\ell$-tensor, then $\Pi^{(\ell)}\sigma_\ell$ is a parallel section of 
$\boxtimes^\ell \cA$, by Propositions \ref{prolong} and \ref{PiBoxtimes}. 
It remains to prove that, if $F$ is a non-vanishing parallel section of
$\boxtimes^\ell\cA$, then $F = \Pi^{(\ell)} \si_\ell$ for some 
Killing $\ell$-tensor $\si_\ell$. 

As a section of $\Ga(S^\ell \cA)$, $F$ has the form
\begin{eqnarray} \label{Fsum}
\begin{split}
F^{\form{A}_1 \ldots \form{A}_\ell} = \sum_{i=0}^\ell
\Y^{(\form{A}_1}_{\,a_1} \ldots \Y^{\form{A}_i}_{\,a_i}
\Z^{\form{A}_{i+1}}_{\,\form{c}_{i+1}} \ldots 
\Z^{\form{A}_\ell)}_{\,\form{c}_\ell}
(\si_i)^{a_1 \ldots a_i \form{c}_{i+1} \ldots \form{c}_\ell},
\end{split}
\end{eqnarray}
where $(\form{A}_1 \ldots \form{A}_\ell)$ denotes the symmetrization over the form tractor
indices (and \idx{not} over the standard tractor indices). Here
$(\si_i)^{a_1 \ldots a_i \form{c}_{i+1} \ldots \form{c}_\ell} \in
\cE^{a_1 \ldots a_i \form{c}_{i+1} \ldots \form{c}_\ell}$
where $a_i$ are indices of the tangent bundle whereas 
$\form{c}_i = [c_i^1c_i^2]$ are form indices. 
Since $F\in\Ga(\boxtimes^\ell \cA)$, the skew symmetrization over any triple of indices of 
$(\si_i)^{a_1 \ldots a_i [c_{i+1}^1c_{i+1}^2] \ldots [c_\ell^1 c_\ell^2]}$ 
vanishes. 

First, we show that $\si_\ell =0$ implies $F= 0$. 
To do that, we assume $\si_{i_0+1} = \ldots = \si_\ell=0$ and prove
$\si_{i_0}=0$, with $0\leq i_0<\ell$. The tractor form
$\na^b F^{\form{A}_1 \ldots \form{A}_\ell}$ can be written in the form
\nn{Fsum}, and it follows from \nn{tracconYZ} that
\begin{align*}
\na^b F^{\form{A}_1 \ldots \form{A}_\ell} =& 2(\ell-i_0) 
\Y^{(\form{A}_1}_{\,a_1} \ldots \Y^{\form{A}_{i_0+1}}_{\,a_{i_0+1}}
\Z^{\form{A}_{i_0+2}}_{\,\form{c}_{i_0+2}} \ldots 
\Z^{\form{A}_\ell)}_{\,\form{c}_\ell}
(\si_{i_0})^{(a_1 \ldots a_{i_0} a_{i_0+1})b
\form{c}_{i_0+2} \ldots \form{c}_\ell} \\
& \text{+ terms with at most $i_0$ of $\Y$'s}.
\end{align*}
Thus $(\si_{i_0})^{(a_1 \ldots a_{i_0} a_{i_0+1})b
\form{c}_{i_0+2} \ldots \form{c}_\ell}=0$. On the other hand, symmetries of
$F$ imply that symmetries of
$(\si_{i_0})^{a_1 \ldots a_{i_0} \form{c}_{i_0+1} \ldots \form{c}_\ell}$
correspond to the Young diagram 
$\raisebox{-13pt}{$
\overbrace{\begin{picture}(60,30)
\put(0,5){\line(1,0){60}}
\put(0,15){\line(1,0){60}}
\put(0,25){\line(1,0){60}}
\put(0,5){\line(0,1){20}}
\put(10,5){\line(0,1){20}}
\put(20,5){\line(0,1){20}}
\put(50,5){\line(0,1){20}}
\put(35,10){\makebox(0,0){$\cdots$}}
\put(35,20){\makebox(0,0){$\cdots$}}
\end{picture}}^{i_0}
\!
\overbrace{\begin{picture}(60,30)
\put(0,15){\line(1,0){60}}
\put(0,25){\line(1,0){60}}
\put(0,5){\line(0,1){20}}
\put(10,15){\line(0,1){10}}
\put(20,15){\line(0,1){10}}
\put(50,15){\line(0,1){10}}
\put(60,15){\line(0,1){10}}
\put(35,20){\makebox(0,0){$\cdots$}}
\put(62,17){\makebox(0,0)[l]{${}$}}
\end{picture}}^{\ell-i_0}$  }$.
Hence $(\si_{i_0})^{(a_1 \ldots a_{i_0} a_{i_0+1})b
\form{c}_{i_0+2} \ldots \form{c}_\ell}=0$ means
$(\si_{i_0})^{a_1 \ldots a_{i_0} \form{c}_{i_0+1} \ldots \form{c}_\ell}=0$,
as wanted.

Next we need to show that the tensor field 
$(\si_\ell)^{a_1 \ldots a_\ell}$ is Killing. Similarly as above, computing the 
$\Y^{(\form{A}_1}_{\,a_1} \ldots \Y^{\form{A}_{\ell})}_{\,a_{\ell}}$-summand
of $\na^b F^{\form{A}_1 \ldots \form{A}_\ell}$ (which is zero), one easily
concludes that $\na^{(b} (\si_\ell)^{a_1 \ldots a_\ell)}=0$. Details are left
to the reader. Finally, since the difference 
$F - \Pi^{(\ell)} \si_\ell \in \Ga(\boxtimes^\ell\cA)$
is parallel and the 
$\Y^{(\form{A}_1}_{\,a_1} \ldots \Y^{\form{A}_{\ell})}_{\,a_{\ell}}$-summand
of $F - \Pi^{(\ell)} \si_\ell$ vanishes, it follows from the first part of 
the proof that $F - \Pi^{(\ell)} \si_\ell =0$.
\end{proof}


\subsection{Construction of commuting symmetries.}
Let $V^{a_1 \ldots a_\ell} \in \cE^{(a_1 \ldots a_\ell)}$ be a symmetric tensor
and $\Pi^{(\ell)}$ the map defined in \nn{Piell}.
We define the differential operator ${\fam2 D}^{V}$ of order $\ell$ by
\begin{equation} \label{cS}
{\fam2 D}^{V} := \langle \Pi^{(\ell)}V, \D^{(\ell)} \rangle: \Gamma(U) \to \Gamma(U),
\end{equation}
where $\D^{(\ell)}: \Gamma(U) \to \bigotimes^\ell \Ga (\cA^*) \otimes \Ga (U)$ 
is the $\ell$th iteration of the operator \nn{doubleD}.
\begin{lemma}\label{lem:pcpalsymbol} 
The differential operator ${\fam2 D}^{V}$ has principal
symbol $V^{a_1 \ldots a_\ell}$.
\end{lemma}
\begin{proof}
Extending the vertical notation for elements in
$\Ga(\cA) = \begin{smallmatrix} \cE^a \\ \oplus \\ \cE^{[ab]}  \end{smallmatrix}$
 and $\Ga(\cA^*) = \begin{smallmatrix} \cE_{[ab]} \\ \oplus \\ \cE_a  \end{smallmatrix}$,
to sections in the tensor products $S^\ell \cA$ and  $S^\ell \cA^*$, 
we obtain
\begin{eqnarray}\nonumber
\Pi^{(\ell)}V = \begin{matrix} 
V^{a_1 \ldots a_\ell} \\ \oplus \\ \vdots
\end{matrix} \in 
\begin{matrix} 
\cE^{a_1 \ldots a_\ell} \\ \oplus \\ \vdots
\end{matrix} \quad \text{and} \quad 
\D^{(\ell)}u = \begin{matrix} 
\vdots \\ \oplus \\ \na_{a_1} \!\ldots\! \na_{a_\ell} u
\end{matrix} \in 
\begin{matrix} 
\vdots \\ \oplus \\ \cE_{a_1 \ldots a_\ell} \otimes \Ga(U).
\end{matrix}
\end{eqnarray}
Thus the contraction $\langle \Pi^{(\ell)}V, \D^{(\ell)} u\rangle$
has the leading term $V^{a_1 \ldots a_\ell} \na_{a_1} \ldots \na_{a_\ell}u$. 
\end{proof}
We consider a Riemannian invariant linear differential 
operator $F: \Ga(U) \to \Ga(U)$, acting on a tensor bundle $U$. 
\begin{theorem} \label{symmU}
Let $k^{a_1 \ldots a_\ell} \in \cE^{(a_1 \ldots a_\ell)}$ be a Killing tensor.
Then the differential operator ${\fam2 D}^{k}$ is a commuting symmetry of $F$ with principal
symbol $k^{a_1 \ldots a_\ell}$.
\end{theorem}

\begin{proof}
By Lemma \ref{lem:pcpalsymbol}, ${\fam2 D}^{k}$ has principal
symbol $k^{a_1 \ldots a_\ell}$.

Let $u \in \Ga(U)$ and $K := \Pi^{(\ell)} k \in \Ga(\bigotimes^\ell \cA)$ be the
prolongation of the Killing $\ell$-tensor $k$. Then we get
$$
F {\fam2 D}^k u= F \langle K, \D^{(\ell)} u\rangle
= \langle K, F^\na \D^{(\ell)} u\rangle 
= \langle K, \D^{(\ell)} F u\rangle = {\fam2 D}^k F u,
$$
where we have used Proposition \ref{prolong} (which implies $\na K=0$) 
in the second equality and Corollary \ref{dDcomm} in the third equality.
Recall the operator $F^\na$ is given by the same formula as $F$ but $\na$ is
interpreted as the coupled Levi-Civita--tractor connection in $F^\na$.
\end{proof}

\begin{corollary}\label{Cor:SymLaplace}
Let us assume $k^{a_1 \ldots a_\ell} \in \cE^{(a_1 \ldots a_\ell)}$ is a Killing 
tensor. Then ${\fam2 D}^k$ is a commuting symmetry of the Laplacian 
$\De: \cE \to \cE$ with principal symbol $k^{a_1 \ldots a_\ell}$.
\end{corollary}

\subsection{Algebra structure on the space of commuting symmetries of the Laplace operator 
$\De: \cE \to \cE$}
\label{salgebra}
Let $\cB$ be the algebra of commuting symmetries of $\De$.
The Theorems \ref{th1} and \ref{th2} allow us to identify the vector 
space of commuting symmetries of $\De$,
\begin{equation} \label{algebra}
\cB\simeq \bigoplus_{\ell=0}^\infty\cK_\ell.
\end{equation}

In order to study the algebra structure on $\cB$, some basic 
notation is needed. Depending on the curvature, the Lie group of isometries is 
$G=SO(p+1,q)$, $G=SO(p,q+1)$ or $G=E(p,q)$. 
For all possibilities, we denote by 
$\gog=Lie(G)$ the Lie algebra of isometries and use the 
identifications $\frak{so}(p+1,q)\simeq \wedge^2\R^{p+1,q}$,
$\frak{so}(p,q+1)\simeq \wedge^2\R^{p,q+1}$, and 
$Lie(E(p,q))\simeq \wedge^2(\R^{p,q} \lpl\R )$.
In the last case, the identification is deduced from the 
representation of $E(p,q)$ on $\R^{p,q}\oplus\R$, induced by the standard group morphism
$\mathrm{GL}(n,\mathbb{R})\ltimes\mathbb{R}^n\rightarrow\mathrm{GL}(n+1,\mathbb{R})$.

The space of parallel sections of $\cA$ is isomorphic to $\gog$ and it is easy 
to verify that the Lie bracket on $\mathfrak{g}$ is isomorphic to the bracket \nn{bracket}.  
Via the induced identification of the symmetric product $S^\ell\gog$ with parallel sections of $S^\ell\cA$, 
we define the subspace $\boxtimes^\ell \gog \subseteq   S^\ell\gog$
 as follows 
$$
\boxtimes^\ell\mathfrak{g}\cong\{\,\text{parallel sections of }\;\boxtimes^\ell\cA\,\},
$$
where $\boxtimes^\ell\cA$ is defined in \nn{boxtimesA}. 
From Theorem \ref{Cor:KT-Young}, we deduce that $\boxtimes^\ell\mathfrak{g}$ is isomorphic ,
as $\mathfrak{g}$-module, to the space $\cK_\ell$ of Killing $\ell$-tensors. Hence 
$\cB\simeq \bigoplus_{\ell=0}^\infty\boxtimes^\ell\gog$ as $\mathfrak{g}$-module.
\begin{theorem}
The symmetry algebra $\cB$ is isomorphic to the tensor algebra 
\begin{eqnarray}
\bigoplus_{i=0}^\infty\otimes^i\gog
\end{eqnarray}
modulo the two-sided ideal $\cI$, generated by 
\begin{eqnarray} \label{ideal}
V\otimes W-V \boxtimes W-\frac{1}{2}[V,W],\quad V,W\in \gog.
\end{eqnarray}
\end{theorem}

\begin{proof}
First we compute the compositions 
${\fam2 D}^k  {\fam2 D}^{\check{k}}$, where $k^a,\check{k}^a \in \Ga(TM)$ are Killing 
vector fields. Put $K = \Pi k^a$ and $\check{K} = \Pi \check{k}^a$ where
$K, \check{K} \in \gog$. Since $\nabla\check{K}=0$, the definition \nn{doubleD12} 
of $\D_\form{B}$ acting on functions yields
\begin{align*}
{\fam2 D}^k {\fam2 D}^{\check{k}} &= K^\form{B} \D_\form{B} 
\check{K}^\form{C} \D_\form{C} \\
&=  
\bigl[ \frac{1}{2} \bigl( 
K^\form{B} \check{K}^\form{C} + K^\form{C} \check{K}^\form{B} \bigr)
+ \frac{1}{2} \bigl( 
K^\form{B} \check{K}^\form{C} - K^\form{C} \check{K}^\form{B} \bigr)
\bigr] \D_\form{B}\D_\form{C} \\
&= \bigl( K \boxtimes \check{K} \bigr)^{\form{BC}} \D_\form{B} \D_\form{C}
+ \frac{1}{2} [K,\check{K}]^\form{B} \D_\form{B}.
\end{align*} 
In the last equality, to deal with the symmetrized term, we use the decomposition 
$$S^2\bigl( \myng{1,1} \bigr)= \myng{2,2} \oplus \myng{1,1,1,1}\;,$$ 
and the identity $\D_{[B^1B^2} \D_{C^1C^2]}=0$, which can be easily verified. 
To deal with the skew-symmetrized term, we use \nn{bracket}.

The computation of ${\fam2 D}^k {\fam2 D}^{\check{k}}$ shows that all elements of the form \nn{ideal}
are in the ideal. Since there is a vector space isomorphism
$$
\bigl( \bigoplus_{\ell=0}^\infty\otimes^\ell\gog \bigr) / \cI \, \cong\,
\bigoplus_{\ell=0}^\infty 
\boxtimes^{\ell}\mathfrak{g}\,\, ,
$$
it remains to show that elements in  
$\boxtimes^{\ell}\mathfrak{g}\cong\cK_\ell$
indeed give rise to non-zero $\ell$th order symmetries. This follows from Corollary \ref{Cor:SymLaplace}.
The proof is complete.
\end{proof}

The passage from the tensor algebra to the universal enveloping algebra
$\cU(\gog)$ 
means to substitute $V\otimes W=\frac{1}{2}(V\otimes W+W\otimes V)-\frac{1}{2}(V\otimes W-W\otimes V)$
and quotient through the two-sided ideal generated by $V\otimes W-W\otimes V=[V,W]$, $V,W\in\gog$.
Accordingly, we get 

\begin{corollary}
The symmetry algebra $\cB$ is isomorphic to the universal enveloping algebra 
$\cU(\gog)$
modulo the two-sided ideal generated by 
\begin{eqnarray}
V\otimes W+W\otimes V-2V\boxtimes W,\quad V,W\in \gog.
\end{eqnarray}
\end{corollary}

\subsection{Examples of commuting symmetries}

The recursion tractor formula \nn{cS} for commuting symmetries ${\fam2 D}^k$ can be 
transformed into an explicit formula for ${\fam2 D}^k$, expressed in terms of the 
Levi-Civita connection $\na$ and the curvature $\rJ$, by \nn{step} and \nn{doubleD12}.
In what follows we compute the explicit commuting symmetries up to order 3 acting
on $\cE$. 

We use tractor form indices
$\form{A} = [A^1A^2]$, $\form{B} = [B^1B^2]$, $\form{C} = [C^1C^2]$ and form indices
$\form{a} = [a^1a^2]$, $\form{b} = [b^1b^2]$, $\form{c} = [c^1c^2]$.

For a Killing vector field $k^a \in \cE^a$, we have 
\begin{equation}
K^\form{A} = (\Pi k)^\form{A} = \Y^\form{A}_{\,a} k^a 
+\frac12 \Z^\form{A}_{\,\form{a}} \na^{[a^1} k^{a^2]},\quad
\D^{}_\form{A} f = 2\Y_\form{A}^{\,a} \na_a f.
\end{equation}
Hence, the symmetry ${\fam2 D}^k f= K^\form{A} \D^{}_\form{A} f = k^a \na_a f$ coincides with
the Lie derivative along the Killing vector field $k^a$. 

For a Killing 2-tensor $k^{bc} \in \cE^{(bc)}$, we get
from \nn{formula}
\begin{eqnarray}
\begin{split}
K^\form{AB} = &(\Pi^{(2)} k)^\form{BC} = 
\Y^\form{B}_{\,b} \Y^\form{C}_{\,c} k^{bc} 
+ \frac{1}{3} \left( \Y^\form{B}_{\,b} \Z^\form{C}_{\,\form{c}} 
+ \Y^\form{C}_{\,b} \Z^\form{B}_{\,\form{c}} \right)
\na^{c^1} k^{c^2b}  \\
& + \frac12  
\Z^\form{B}_{\,\form{b}} \Z^\form{C}_{\,\form{c}} \left[
\frac13 \na^{b^1} \na^{c^1} k^{b^2c^2}
+ \frac2n \rJ g^{b^1c^1} k^{b^2c^2} \right].
\end{split}
\end{eqnarray}
Since $\D_\form{B} \D_\form{C} f 
= 4 \Y_\form{B}^{\,b} \Z_\form{C}^{\,\form{c}} g_{bc^0}\na_{c^1}f
+ 4 \Y_\form{B}^{\,b} \Y_\form{C}^{\,c}\na_b \na_c f$ by \nn{doubleD12}, 
we obtain
\begin{eqnarray*}
{\fam2 D}^k = K^\form{BC} \D_\form{B} \D_\form{C} f = 
k^{bc}\na_b\na_c f + (\na_r k^{rc})\na_c f.
\end{eqnarray*}
Note that $K^\form{BC} h_\form{BC}$ is a constant, and  
using $\na^a k^r{}_r = -2 \na_r k^{ar}$ (which follows from
$3 g_{bc}\na^{(a}k^{bc)}= \na^a k^r{}_r + 2\na_r k^{ar}=0$), 
a short computation reveals its value 
\begin{eqnarray*}
K^\form{BC} h_\form{BC} = \frac14 [-\na_r\na_s \si^{rs} + 
\frac{2(n+1)}{n} \rJ k^r{}_r].
\end{eqnarray*}
Thus the modification of ${\fam2 D}^k$ by any multiple of 
$\na_r\na_s k^{rs} - \frac{2(n+1)}{n} \rJ k^r{}_r$ is again a symmetry of $\Delta$.
This means there is no unique formula, written in terms of $k^{ab}$, for a symmetry.
This is in contrast with the case of conformal symmetries  \cite{east}. 

Now we consider Killing 3-tensors $k^{abc} \in \cE^{(abc)}$. Then
\begin{eqnarray*}
g_{cd}\na^{(a} k^{bcd)} = 3\na_r k^{rab} + 3 \na^a k^{br}{}_r =0,
\end{eqnarray*}
and applying $\na_a$ we obtain $\na_r\na_s k^{rsa} + \De k^{ar}{}_r=0$. 
Summarizing, we obtain
\begin{equation} \label{3-Kill}
\na_r k^{rab} = -\na^a k^{br}{}_r, \quad 
\na_r\na_s k^{rsa} = -\De k^{ar}{}_r \quad \text{and} \quad
\na_r k^{rs}{}_s=0,
\end{equation}
where the last equality is the trace of 
$\na_r k^{rab} + \na^a k^{br}{}_r =0$. Now computing $\Pi^{(3)}k$
(which requires to apply $\Pi$ in \nn{Pi} to \nn{formula}) results in
\begin{eqnarray*}
\begin{split}
&K^\form{ABC} = (\Pi^{(3)} k)^\form{ABC} = 
\Y^\form{A}_{\,a} \Y^\form{B}_{\,b} \Y^\form{C}_{\,c} k^{abc}  \\ 
&\ \ + \frac{1}{4} \left(  
\Y^\form{A}_{\,a} \Y^\form{B}_{\,b} \Z^\form{C}_{\,\form{c}} 
+ \Y^\form{C}_{\,a} \Y^\form{A}_{\,b} \Z^\form{B}_{\,\form{c}} 
+ \Y^\form{B}_{\,a} \Y^\form{C}_{\,b} \Z^\form{A}_{\,\form{c}}  \right)
\na^{c^1} k^{c^2ab} \\
&\ \ + \frac13 \left(
\Y^\form{A}_{\,a} \Z^\form{B}_{\,\form{b}} \Z^\form{C}_{\,\form{c}} 
+ \Y^\form{C}_{\,a} \Z^\form{A}_{\,\form{b}} \Z^\form{B}_{\,\form{c}} 
+ \Y^\form{B}_{\,a} \Z^\form{C}_{\,\form{b}} \Z^\form{A}_{\,\form{c}} \right)
\left[
\frac14 \na^{b^1} \na^{c^1} k^{b^2c^2a}
+ \frac2n \rJ g^{b^1c^1} k^{b^2c^2a} \right] \\
&\ \ + 
\Z^\form{A}_{\,\form{a}} \Z^\form{B}_{\,\form{b}} \Z^\form{C}_{\,\form{c}} 
\psi_\form{abc}
\end{split}
\end{eqnarray*}
for some $\psi_\form{abc}$ which we do not need to compute.
Furthermore,
\begin{align}
\D_\form{A} &\D_\form{B} \D_\form{C} f 
= 8 \Y_\form{A}^{\,a} \Z_\form{B}^{\,\form{b}} \Z_\form{C}^{\,\form{c}} 
g_{ab^1} g_{b^2c^1} \na_{c^2} f \nonumber  \\
& + 16 \Y_\form{A}^{\,a} \Y_\form{B}^{\,b} \Z_\form{C}^{\,\form{c}} 
g_{c^1(a} \na_{b)} \na_{c^2} f
+ 8 \Y_\form{A}^{\,a} \Z_\form{B}^{\,\form{b}} \Y_\form{C}^{\,c}
g_{ab^1} \na_{c} \na_{b^2} f \nonumber \\
& + 8 \Y_\form{A}^{\,a} \Y_\form{B}^{\,b} \Y_\form{C}^{\,c} 
\left[ \na_a \na_b \na_c f - \frac{4}{n} \rJ g_{b[a} \na_{c]} f \right]
\end{align}
by \nn{doubleD12}. Combining the two previous displays yields
\begin{align}
{\fam2 D}^k = K^\form{ABC} \D_\form{A} \D_\form{B} \D_\form{C} f =& 
k^{abc}\na_a\na_b\na_c f + \frac{3}{2} (\na_r k^{rbc})\na_b\na_c f  \nonumber \\
& + \frac14 (\na_r\na_s k^{rsc}) \na_cf - \frac{n-1}{2n} \rJ k^{cr}{}_r\na_cf.
\end{align}
By construction, the vector field $(\Pi^{(2)}k)^{a\form{BC}}h_\form{BC}$
is Killing. Using \nn{formula}, \nn{tracmetricforms} and \nn{3-Kill}, one 
easily computes
$$
(\Pi^{(2)}k)^{a\form{BC}}h_\form{BC}  = 
-\frac{1}{12} \bigl[ \na_r\na_s k^{rsa} - \frac{4(n+2)}{n} \rJ k^{ar}{}_r
\bigr].
$$ 
Thus the symmetry ${\fam2 D}^k$ can be modified by a multiple of the operator 
$[\na_r\na_s k^{rsa} - \frac{4(n+2)}{n} \rJ k^{ar}{}_r] \na_af$.

\begin{remark}
Let $\ell\in\mathbb{N}$ and $k\in\cK_\ell$. If the curvature of the metric $g$ vanishes, 
i.e. $M$ is locally isomorphic to the pseudo-Euclidean space $\mathbb{R}^{p,q}$, 
a straightforward computation shows that
$$
{\fam2 D}^k=\sum_{i=0}^\ell \frac{1}{2^i}\binom{\ell}{i} 
\left(\nabla_{a_1}\cdots \nabla_{a_i} k^{a_1\cdots a_\ell}\right)\nabla_{a_{i+1}}\cdots \nabla_{a_\ell}
$$
is a commuting symmetry of $\Delta$. 
This can also be deduced from properties of the Weyl 
quantization of $T^*\mathbb{R}^{p,q}$.
Namely, ${\fam2 D}^k$ coincides with the Weyl quantization of $k$, and the symplectic 
equivariance of the Weyl quantization (see, e.g., 
\cite{Fol89}) yields the equalities $[\Delta,{\fam2 D}^k]=[g,k]_S=0$. Here, 
$[\cdot,\cdot]_S$ denotes the Schouten bracket of symmetric tensors and 
$[g,k]_S=0$ is equivalent to the Killing equation.
\end{remark}

\begin{remark}
Let $\ell\in\mathbb{N}$ and $k\in\cK_\ell$. If $k$ is trace-free, then a straightforward
computation shows that
$$
{\fam2 D}^k=k^{a_1\cdots a_\ell}\nabla_{a_{1}}\cdots \nabla_{a_\ell}
$$
is a commuting symmetry of $\Delta$. This can also be deduced 
from the work of Eastwood \cite{east}. 
Out of conformal Killing tensors $V$, he explicitly builds conformal symmetries of the Laplacian,
i.e., differential operators ${\fam2 D}_1^V$ and ${\fam2 D}_2^V$ with principal symbol $V$ such that
${\fam2 D}_2^V\Delta=\Delta{\fam2 D}_1^V$. The lower order terms involve divergences of $V$ 
and contractions of $V$ with the trace-free Ricci tensor. 
On a space of constant curvature, with $V=k$ a trace-free Killing tensor, both vanish 
and we get ${\fam2 D}_1^k={\fam2 D}_2^k={\fam2 D}^k$.
\end{remark}


\section{Riemannian geometry via projective geometry}
\label{sectproj}

Overdetermined equations for Killing tensors are 
projectively invariant \cite{EGproj}, so it is natural
to consider their prolongation within the framework of projective geometry.
As this is an example of a parabolic geometry, we can employ the general
invariant theory for this class of structures, \cite{CSbook}. 
We shall observe that several results obtain in the 
previous section then follow immediately.

Recall we are interested in manifolds of constant curvature. 
These are conformally
flat and thus projectively flat as well (see \eqref{Rproj} below).  
That is, we will consider locally flat projective structures. 

\subsection{Tractor calculus in projective geometry}
We shall briefly recall invariant calculus on projective manifolds, see
\cite{beg} for more details.
A projective structure on a manifold $M$ is given by a class $[\nabla]$ of special affine  connections  
with the same geodesics as unparametrized curves ("special" means that 
there is a parallel volume form for every connection in $[\na]$.) These
connections are parametrized by nowhere vanishing sections of projective 
density bundles $\cE(1)$. We shall also assume 
orientability, characterized by a compatible volume form 
$\epsilon_{a^1 \ldots a^n} \in \cE_{[a^1 \ldots a^n]}(n+1)\cong\cE$,
parallel for every affine connection in $[\na]$.
 The decomposition of the curvature of $\na$ is
\begin{equation} \label{Rproj} 
R_{ab}{}^c{}_d = \overline{C}_{ab}{}^c{}_d + 
2\de_{[a}^c \overline{\rP}_{b]d}^{},
\end{equation}
where $\overline{\rP}_{ab}$ is the projective Schouten tensor and 
$\overline{C}_{ab}{}^c{}_d=C_{ab}{}^c{}_d$. (That is, conformal and projective 
Weyl tensors coincide.) Note that for the Levi-Civita connection $\na$ of 
an Einstein metric $g$, the curvature is 
also of the form \nn{Rriem} and the relation between projective and Riemannian
 Schouten tensors is 
$\overline{\rP}_{ab} = 2\rP_{ab}$, \cite{gm}. 

We define the standard tractor bundle and its dual by their spaces of
sections $\overline{\cE}^A$ and $\overline{\cE}_A$, respectively, as 
\begin{equation*}
\overline{\cE}^A = \begin{matrix}
\cE^a(-1) \\ \upl \\ \cE(-1)
\end{matrix} \qquad \text{and} \qquad
\overline{\cE}_A = \begin{matrix}
\cE(1) \\ \upl \\ \cE_a(1)
\end{matrix},
\end{equation*}
see \cite{beg} for the meaning of the semi-direct product $\upl$. The choice of 
a connection in the class $[\na]$ turns the previous display into the direct sum
decomposition.
These bundles are equipped with the {\em projectively invariant tractor connection}
which we denote by $\overline{\na}$. Choosing $\na$ in the projective class,
$\overline{\na}$ is explicitly given by the formulas
\begin{equation} \label{tracconproj}
\overline{\nabla}_a\begin{pmatrix}
\nu^b\\ \rho
\end{pmatrix}
=\begin{pmatrix}
\nabla_a \nu^b  + \de_a^b \rho \\
\nabla_a \rho - \overline{\rP}_{ab} \nu^b
\end{pmatrix}
\quad \text{and} \quad
\overline{\nabla}_a\begin{pmatrix}
\si \\ \mu_{b}
\end{pmatrix}
=\begin{pmatrix}
\nabla_a \si  - \mu_a \\
\nabla_a \mu_b + \overline{\rP}_{ab} \si
\end{pmatrix},
\end{equation}
see \cite{beg} for details. Here $\nu^a \in \cE^a(-1)$, $\rho \in \cE(-1)$,
$\si \in \cE(1)$ and $\mu_a \in \cE_a(1)$. We extend the connection $\overline{\nabla}$
to the tensor products of $\overline{\cE}^A$ by the Leibniz rule.
Also note that the structure of the tractor bundle $\overline{\cE}^{[AB]}$
and of its dual is given by
\begin{equation} \label{projad}
\overline{\cE}^{[AB]} =
\begin{matrix}
\cE^{[ab]}(-2) \\ \upl \\ \cE^{a}(-2)
\end{matrix} \qquad \text{and} \qquad
\overline{\cE}_{[AB]} =
\begin{matrix}
\cE_a(2) \\ \upl \\ \cE_{[ab]}(2).
\end{matrix}
\end{equation}

In what follows, we shall use the tractor bundle
\begin{eqnarray}
\overline{\cE}_{A}^{\,\, B} = \overline{\cE}_A \otimes \overline{\cE}^B =
\begin{matrix}
{\cE}^a \\ \upl \\ \cE_{a}^{\,\, b} \oplus \cE \\ \upl \\ \cE_a
\end{matrix},
\end{eqnarray}
where the trace-free part of $\overline{\cE}_{A}^{\,\, B}$ is isomorphic to the 
projective adjoint tractor bundle. Analogously to \nn{doubleD}, we define the 
projectively invariant differential operator
\begin{equation} \label{doubleDproj}
\overline{{\mathbb D}}_A{}^B: 
\cE_{b_1 \ldots b_s}(w) \otimes \overline{\cE}_{C \ldots D}{}^{E \ldots F} \to
\cE_{b_1 \ldots b_s}(w) \otimes \overline{\cE}_A{}^B \otimes 
\overline{\cE}_{C \ldots D}{}^{E \ldots F},
\end{equation}
as follows. Acting on $f\in\cE(w)$ and $\varphi_a\in\cE_a$, $\overline{{\mathbb D}}_A{}^B$
is given by
\begin{eqnarray} \label{dDprojex}
\overline{{\mathbb D}}_A{}^{B} f=
\begin{pmatrix}
0 \\ 0 \,|\, wf \\ \nabla_a f
\end{pmatrix}, \quad
\overline{{\mathbb D}}_A{}^{B}\varphi_c^{}=
\begin{pmatrix}
0 \\ \delta_{c}^b\varphi_{a}^{} \,| -\varphi_c \\ \nabla_a \varphi_c
\end{pmatrix},
\end{eqnarray}
for an affine connection $\na$ in the projective class. 
The formula for 
$\overline{{\mathbb D}}_A{}^B f$ for 
$f \in \overline{\cE}_{C \ldots D}{}^{E \ldots F}$ is formally the same as for 
$f\in\cE(0)$ where we interpret $\na$ as the coupled affine--tractor 
connection. Then we extend $\overline{{\mathbb D}}_A{}^B$
to the general case by the Leibniz rule. 

Henceforth,
we assume the manifold $M$ is projectively flat, i.e.,\ the projective Weyl 
tensor vanishes. This in particular means that 
the tractor connection $\overline{\na}$ is flat.  

Let $F: \Ga(U_1) \to \Ga(U_2)$ be a projectively invariant linear differential 
operator, acting between tensor bundles $U_1$, $U_2$. Then, $F$ can be written in terms
of an affine connection $\na$. 
Regarding $\na$ in the formula for $F$ as the 
coupled affine--tractor connection, we obtain the operator 
$F^{\overline{\na}}: \Ga(\cA^* \otimes U_1) \to \Ga(\cA^* \otimes U_2)$. 
Adapting the proof of Theorem 
\ref{Dcomm} to the projective setting, we obtain the analogue of
Corollary \ref{dDcomm}:

\begin{theorem} \label{dDcommproj}
Let $F: \Ga(U_1) \to \Ga(U_2)$ be a projectively invariant linear differential 
operator over a projectively flat manifold. Then $\overline{\D}_A{}^B$ commutes 
with $F$, i.e.\ 
$$
\overline{\D} \circ F = F^{\overline{\na}} \circ \overline{\D}: \Ga(U_1) \to 
\overline{\cE}_A{}^B \otimes \Ga(U_2).
$$
\end{theorem}
As an example, consider the projectively invariant differential operator
\begin{eqnarray} \label{inv2}
\nabla_{(a}\nabla_{b)}+\overline{\rP}_{ab}: \cE(1)\to\cE_{(ab)}(1),
\end{eqnarray}
see e.g.\ \cite{alg}.  
The projective invariance and Theorem \ref{dDcommproj} imply 
\begin{eqnarray} \label{dDcommeq}
\overline{{\mathbb D}}_{A_1}^{\,\, B_1} \dots 
\overline{{\mathbb D}}_{A_\ell}^{\,\, B_{\ell}}
\bigl( \nabla_{(a} \nabla_{b)} + \overline{\rP}_{ab} \bigr)=
\bigl( \nabla_{(a}\nabla_{b)} + \overline{\rP}_{ab} \bigr) 
\overline{{\mathbb D}}_{A_1}^{\,\, B_1} \dots 
\overline{{\mathbb D}}_{A_\ell}^{\,\, B_{\ell}},
\end{eqnarray}   
where $\na$ on the right hand side denotes the coupled affine-tractor connection.

\subsection{Killing tensors in projective geometry}

Let $\ell\in\mathbb{N}$. We shall focus on the PDE
\begin{eqnarray}\label{projsymm}
\nabla_{(a_0}k_{a_1\dots a_\ell)}=0, \quad 
k_{a_1\dots a_\ell}\in\cE_{(a_1\dots a_\ell)}(2\ell), 
\end{eqnarray}
which is projectively invariant \cite{EGproj}. 

Putting $\overline{\cE}_A = \myng{1}$, we have $\overline{\cE}_{[AB]} = \myng{1,1}$ and we set 
$$
\boxtimes^\ell\overline{\cE}_{[AB]} := 
\raisebox{-13pt}{$
\overbrace{\begin{picture}(60,30)
\put(0,5){\line(1,0){60}}
\put(0,15){\line(1,0){60}}
\put(0,25){\line(1,0){60}}
\put(0,5){\line(0,1){20}}
\put(10,5){\line(0,1){20}}
\put(20,5){\line(0,1){20}}
\put(50,5){\line(0,1){20}}
\put(60,5){\line(0,1){20}}
\put(35,10){\makebox(0,0){$\cdots$}}
\put(35,20){\makebox(0,0){$\cdots$}}
\end{picture}}^{\ell}$}
\subseteq  S^\ell \overline{\cE}_{[AB]}.
$$
There exists a linear map 
$\overline{\Pi}^{(\ell)}:\cE_{a_1\ldots a_\ell}(2\ell)\rightarrow\boxtimes^\ell \overline{\cE}_{[AB]}$, 
characterized by curved Casimir operators (see \cite{CScas}), which takes the form
\begin{equation} \label{kK}
\overline{\Pi}^{(\ell)}: k_{a_1 \ldots a_\ell} \mapsto \overline{K}_{[A_1B_1] \ldots [A_\ell B_\ell]}
= \begin{matrix}
k_{a_1 \ldots a_\ell} \\ \upl \\ \vdots
\end{matrix}
\in \begin{matrix}
\cE_{a_1 \ldots a_\ell}(2\ell) \\ \upl \\ \vdots \, 
\end{matrix} 
\end{equation}
and such that $k_{a_1 \ldots a_\ell}$ is 
a solution of \nn{projsymm} if and only if 
$\overline{K}_{[A_1B_1] \ldots [A_\ell B_\ell]}$ is $\overline{\nabla}$-parallel.
Note the unspecified terms (indicated by vertical dots) of 
$\overline{K}_{[A_1B_1] \ldots [A_\ell B_\ell]}$ are differential in 
$k_{a_1 \ldots a_\ell}$, i.e.,\ the map $\overline{\Pi}^{(\ell)}$
 is given by a differential operator.
In fact, this is an example of a splitting operator (see e.g. \cite{CScas} for details). 
It yields an analog of Theorem \ref{Cor:KT-Young}. 
\begin{proposition} \cite{BCEG}
Let $(M,[\na])$ be a projectively flat manifold.
The map $\overline{\Pi}^{(\ell)}$ induces a bijective correspondence  
\begin{equation} \label{projKilling}
\{\text{solutions}\ k_{a_1 \ldots a_\ell}\ \text{of \nn{projsymm}} \}
\stackrel{1-1}{\longleftrightarrow} 
\{\overline{\nabla}-\text{parallel sections of}\ \boxtimes^\ell \overline{\cE}_{[AB]}\}.
\end{equation} 
\end{proposition}
If the  Levi-Civita connection of a metric $g$ pertains to the projective class $[\na]$, 
the latter proposition gives a description of Killing tensors for the metric $g$ via the map 
\begin{equation}\label{MetricMap}
V^{(a_1\ldots a_\ell)}\mapsto g_{a_1b_1}\cdots g_{a_1b_1}V^{(a_1\ldots a_\ell)}\in\cE_{(b_1\ldots b_\ell)}(2\ell).
\end{equation}
Indeed, this map gives a bijection between Killing $\ell$-tensors and solutions of the Equation \nn{projsymm}.

\subsection{Construction of symmetries}

Now assume there is a Levi-Civita connection $\na$ in the projective class $[\na]$, such that
the associated metric $g_{ab}$ has constant curvature, i.e.,\
$R_{abcd} = \frac{4}{n}\rJ g_{c[a}g_{b]d}$ with $\rJ$ parallel. 
Then a short computation based on \nn{tracconproj} and \nn{traccon} shows that 
$\overline{\na}$ on  
$\overline{\cE}^{[AB]}$ (resp.\ $\overline{\cE}_{[AB]}$) agrees with the
Riemannian tractor connection $\na$ on  $\cE^\form{A}$ (resp.\ $\cE_\form{A}$).
Moreover the tractor section
\begin{equation}\label{h}
h^{AB} = \begin{pmatrix}
g^{ab} \\ 0 \\ \frac2n \rJ
\end{pmatrix}
\in \overline{\cE}^{(AB)} = 
\begin{matrix}
\cE^{(ab)}(-2) \\ \upl \\ \cE^a(-2) \\ \upl \\ \cE(-2) 
\end{matrix} \cong
\begin{matrix}
\cE^{(ab)} \\ \oplus \\ \cE^a \\ \oplus \\ \cE 
\end{matrix}
\end{equation}
is parallel, cf.\ \cite{gm}. The isomorphism $\cong$ corresponds to the choice of 
connection $\na\in[\na]$,
and in particular trivializes density bundles. 
Here $g^{ab}$ is the inverse of $g_{ab}$ and $\rJ = g^{ab} \rP_{ab} = 
\frac12 g^{ab} \overline{\rP}_{ab}$. Summarizing, we shall consider
the Riemannian manifold $(M,g)$ as the corresponding locally  flat
projective manifold $(M,[\na])$ with the distinguished parallel section
$h^{AB}$. 
This is an example of
 holonomy reduction of Cartan connections, \cite{CGH},
for the projective Cartan connection associated to $(M,[\na])$.

For $\rJ \not=0$, note $h^{AB}$ is non-degenerate, hence a tractor metric.
A direct computation gives the following display and lemma: 
\begin{equation} \label{dDg}
h^{P[A}\overline{{\mathbb D}}_P{}^{B]} g^{ab} = 
h^{P[A}\overline{{\mathbb D}}_P{}^{B]} g_{ab} = 0.
\end{equation}
\begin{lemma} \label{indep}
The explicit formula for the differential operator
$$
h^{P[A}\overline{{\mathbb D}}_P{}^{B]}: \cE(w) \to \cE^{[AB]}(w),
$$
written in terms of the Levi-Civita connection $\na$, does not depend 
on 
$w \in \mathbb{R}$. \qed
\end{lemma}

\vspace{1ex}

We are ready now to construct the commuting symmetries of the Laplace operator. 
The metric $g$ allows to identify a tensor $V^{a_1 \ldots a_\ell} \in \cE^{(a_1 \ldots a_\ell)}$
with an element in $\cE_{(a_1 \ldots a_\ell)}(2\ell)$, see \nn{MetricMap}, 
and we denote by $\overline{V}\in\boxtimes^\ell\overline{\cE}_{[AB]}$ 
the corresponding tractor, obtained via the map $\overline{\Pi}^\ell$ (see \nn{kK}).
We consider the operators 
\begin{eqnarray} \label{S}
\overline{{\fam2 D}}^V:=h^{A_1C_1}\dots h^{A_\ell C_\ell }
\overline{V}_{ A_1B_1 \dots A_\ell B_\ell} \overline{\mathbb D}_{C_1}{}^{B_1} 
\dots \overline{\mathbb D}_{C_\ell}{}^{B_{\ell}},
\end{eqnarray}
acting on any tensor-tractor bundle $U$.

\begin{lemma} \label{symbol}
The principal symbol of the differential operator $\overline{{\fam2 D}}^V$ 
is the symmetric $\ell$-tensor $V$.
\end{lemma}

\begin{proof}
The proof is analogous to the proof of Theorem \ref{symmU}.
Writing tractor sections in the vertical notation (see \nn{projad}),
we can refer to their ``top'' or ``bottom'' parts.
The ``top'' part of $\overline{V}_{A_1B_1\dots A_\ell B_\ell}$ is 
$g_{a_1b_1}\cdots g_{a_1b_1}V^{(a_1\ldots a_\ell)}$, 
cf.\ \nn{kK} and \nn{MetricMap}. On the other hand, 
an elementary computation (using \nn{dDprojex} and \nn{h}) shows that 
the ``bottom'' part (and the leading term)
of $h^{C[A}\overline{{\mathbb D}}_C{}^{B]}f$ is 
equal to $g^{ab}\nabla_bf$ for any section $f$ of a tensor bundle. Therefore, 
the ``bottom'' part of the composition 
$h^{C_1[A_1}\overline{{\mathbb D}}_{C_1}{}^{B_1]}\dots h^{C_\ell [A_\ell}
\overline{{\mathbb D}}_{C_\ell}{}^{B_\ell]}f$  
is equal to $g^{a_1b_1}\nabla_{b_1}\dots g^{a_\ell b_\ell}\nabla_{b_\ell}f$. 
This completes the proof of the Lemma.
\end{proof}

\begin{theorem}
Let  $(M,g)$ be a pseudo-Riemannian manifold of constant curvature,
 with Levi-Civita connection $\na$, and $(M,[\na])$ the corresponding 
locally flat projective manifold. Then, if $k$ is a Killing $\ell$-tensor, 
the operator $\overline{{\fam2 D}}^k: \cE \to \cE$, defined by \eqref{S}, 
is a commuting symmetry of the Laplace 
operator $\De = g^{ab}\na_a\na_b$. 
\end{theorem}

\begin{proof} 
Let $\overline{K}\in\boxtimes^\ell \overline{\cE}_{[AB]}$ be the parallel tractor
associated to $k$ via the composition of the maps \nn{MetricMap} and \nn{projKilling}. 
Since the tractor metric $h$ is also parallel with respect to the projective 
tractor connection $\overline{\na}$, it follows from \nn{dDcommeq} and \nn{dDg} that 
$$
\overline{{\fam2 D}}^k \bigl( g^{ab} ( \nabla_{(a}\nabla_{b)}+\overline{\rP}_{ab} ) 
\bigr) = \bigl( g^{ab} ( \nabla_{(a}\nabla_{b)}+\overline{\rP}_{ab}) \bigr) 
\overline{{\fam2 D}}^k: \cE(+1) \to \cE(-1),
$$
where we consider $g^{ab} \in \cE^{(ab)}(-2)$. 
The operator $\overline{{\fam2 D}}^k: \cE(w) \to \cE(w)$,
 expressed in terms of $\na$, does not depend on $w \in \mathbb{R}$
by Lemma \ref{indep}.
Observing that  $g^{ab}\overline{\rP}_{ab}$ is parallel 
for $\na$, the theorem follows.
\end{proof}

\begin{remark} 
Projectively invariant overdetermined operators, as the operator defined in \eqref{inv2},
are discussed in  \cite{EGproj}. They allow for analogous construction of symmetries for other
Riemannian linear differential operators $F: \Ga(U) \to \Ga(U)$.
\end{remark}


\vspace{0.2cm}
{\bf Acknowledgements:} JPM is supported by 
the Belgian Interuniversity Attraction Pole (IAP) within the framework "Dynamics,
Geometry and Statistical Physics" (DYGEST). 
PS and JS gratefully acknowledge the support of the 
grant agency of the Czech Republic under the grant P201/12/G028. 


\end{document}